\documentclass
{article}
\usepackage[latin1]{inputenc}
\usepackage{indentfirst}

\usepackage{amsmath,amsfonts,amssymb,amsthm,amscd}
\RequirePackage{ifthen}

\usepackage{latexsym}
\usepackage{mathrsfs}
\usepackage{algorithm}
\usepackage[noend]{algorithmic}
\usepackage{paralist}
\usepackage{graphicx}
\usepackage{booktabs}
\usepackage{mathtools}
\usepackage{rotating}
\usepackage{tikz}

\usepackage{subcaption}

\usepackage[normalem]{ulem}
\usepackage[color,matrix,arrow]{xy}
\usepackage{mathtools}

\usetikzlibrary{patterns}

\usepackage[a4paper,top=1.45cm,bottom=1.55cm,left=1.59cm,right=1.59cm]{geometry}

\usepackage{hyperref}

\input xy
\xyoption{all}
\newcommand{\leqdr}{\mathbin{\rotatebox[origin=c]{-45}{$\leq$}}}
\newcommand{\lequr}{\mathbin{\rotatebox[origin=c]{45}{$\leq$}}}
\newdir{d}{\leq}
\newdir{d}{\geq}
\newdir{dr}{\leqdr}
\newdir{ur}{\lequr}

\theoremstyle{plain}

\theoremstyle
{plain}
\newtheorem{theorem}{Theorem}[section]

\newtheorem{proposition}[theorem]{Proposition}

\newtheorem{fact}[theorem]{Fact}
\newtheorem{lemma}[theorem]{Lemma}
\newtheorem{corollary}[theorem]{Corollary}

\newtheorem{question}[theorem]{Question}
\newtheorem{claim}[theorem]{Claim}
\theoremstyle{definition}
\newtheorem{definition}[theorem]{Definition}

\newtheorem{example}[theorem]{Example}
\newtheorem{remark}[theorem]{Remark}

\makeindex

\newcommand{\N}{\mathbb{N}}
\newcommand{\Z}{\mathbb{Z}}

\newcommand{\R}{\mathbb{R}}

\newcommand{\XX}{\mathcal X}

\newcommand{\VV}{\mathcal V}
\newcommand{\YY}{\mathcal Y}

\def\Met{\mathbf{Met}}
\def\RMet{\mathbf{RMet}}
\def\Metsym{\mathbf{Met/\!_\sim}}
\def\RMetsym{\mathbf{RMet/\!_\sim}}
\def\QMet{\mathbf{QMet}}
\def\RQMet{\mathbf{RQMet}}
\def\QMetsym{\mathbf{QMet/\!_{\sim_{\Sym}}}}
\def\RQMetsym{\mathbf{RQMet/\!_{\sim_{\Sym}}}}

\DeclareMathOperator{\diam}{diam}

\DeclareMathOperator{\asdim}{asdim}

\DeclareMathOperator{\Mor}{Mor}

\DeclareMathOperator{\Sym}{Sym}

\DeclareMathOperator{\Ffun}{F}
\DeclareMathOperator{\Qfun}{Q}

\DeclareMathOperator{\Ifun}{I}
\DeclareMathOperator{\dis}{dis}
\DeclareMathOperator{\Fix}{Fix}

\DeclareMathOperator{\Funct}{Funct}

\author{Nicol\`o Zava
	\\  \\ {\footnotesize Institute of Science and Technology Austria (ISTA)}\\
	{\footnotesize Am Campus 1, 3400 Klosterneuburg, Austria}\\
	{\footnotesize {\tt nicolo.zava@gmail.com}}}
\title{Stability of the $q$-hyperconvex hull of a quasi-metric space}
\date{}

\begin{document}
	\maketitle
	
	\begin{abstract}
	In this paper, we study the stability of the $q$-hyperconvex hull of a quasi-metric space, adapting known results for the hyperconvex hull of a metric space. To pursue this goal, we extend well-known metric notions, such as Gromov-Hausdorff distance and rough isometries, to the realm of quasi-metric spaces. In particular, we prove that two $q$-hyperconvex hulls are close with respect to the Gromov-Hausdorff distance if so are the original spaces. Moreover, we provide an intrinsic characterisation of those spaces that are $\Sym$-large in their $q$-hyperconvex hulls.
	\end{abstract}
	
\begin{MSC}
54E35, 
51F30, 
53C23, 
54D35. 
\end{MSC}


\section{Introduction}

	The notion of hyperconvex metric space dates back to the work of Aronszajn and Panitchpakdi \cite{AroPan}. A metric space is said to be hyperconvex if it is metrically convex (i.e., two balls that can intersect do intersect) and Helly (i.e. any family of balls that pairwise intersect, has non-empty intersection). Hyperconvex spaces have many relevant properties for which we refer to \cite{EspKha}. Among these, they are contractible, geodesic, the Vietoris-Rips complex coincides with the \v{C}ech complex, and let us also cite the fact that every non-expansive self-map of a bounded hyperconvex space has a fixed point (\cite[Theorem 6.1]{EspKha}, while some previous results in more restrictive settings can be found in \cite{Sin} and \cite{Soa}).
	
	Given a metric space $X$, its hyperconvex hull $E(X)$ consists of a hyperconvex metric space $E(X)$ and an isometric embedding $\mathfrak e\colon X\to E(X)$ such that any other isometric embedding $\varphi\colon X\to Y$ into a hyperconvex space splits through $\mathfrak e$, i.e., there exists a unique isometric embedding $\psi\colon E(X)\to Y$ such that $\psi\circ\mathfrak e=\varphi$. In the literature, the hyperconvex hull is also known as the injective hull or envelope (as it is an injective object in the category of metric spaces and non-expansive maps) or tight span. An explicit construction of the hyperconvex hull can be found in \cite{Isb}, which was later rediscovered in \cite{Dre} and in \cite{ChrLam}. Hyperconvex hulls have proven to be useful in phylogenetic analysis (\cite{DreHubMou,DreMouTer}). Moreover, in \cite{LimMemOku} the authors proved that the standard persistent homology of the Vietoris-Rips filtration of a metric space $X$ is isomorphic to the persistent homology of the filtration obtained by thickening $X$ in its hyperconvex hull (see also \cite{LimMemWanWanZho} for further applications). 
	
	The hyperconvex hull is an unstable construction, and taking equivalent metrics on the same underlying set can cause drastic modifications in the respective hulls (Remark \ref{rem:unstable_hyper_hull}). However, this construction is stable under some other metric equivalences, known as rough isometries (or $\varepsilon$-isometries), in the following sense: if two metric spaces are roughly isometric, then so are their hyperconvex hulls. Since two metric spaces are close in the Gromov-Hausdorff distance if and only if they are roughly isometric (see, for example, \cite{BurBurIva}), the stated stability result follows from \cite{Moe}, where the author proved that, for any pair of metric spaces, the Gromov-Hausdorff distance of the hyperconvex hulls is bounded from above by eight times the distance between the original spaces (in \cite{LanPavZus} the upper bound was reduced to twice the distance, and its tightness was proved). 
	
	In \cite{HaeHodPet}, the authors, inspired by \cite{CheEst}, studied the class of coarsely injective metric spaces in conjunction with other notions of non-positively curved metric spaces, aiming for applications to geometric group theory. As the name suggests, a metric space $(X,d)$ is coarsely injective if and only if $\mathfrak e(X)$ is large in $E(X)$, i.e., $\sup\{\inf\{d(x,y)\mid y\in\mathfrak e(X)\}\mid x\in E(X)\}<\infty$. In \cite{Lan}, a particular class of coarsely injective metric spaces was discussed. Namely, if a metric space $X$ is discretely geodesic and Gromov-hyperbolic, then $X$ is coarsely injective, and, moreover, also $E(X)$ is Gromov-hyperbolic. In \cite{BauRol}, this result was used to provide a bound for the contractibility of the Vietoris-Rips complex of Gromov-hyperbolic $\nu$-geodesic metric spaces. We refer to \cite{BonSch} for a comprehensive discussion of the connectedness properties of Gromov-hyperbolic metric spaces.
	
	In many situations, the symmetry of a metric is not granted, and, intuitively, going from one point to another may have different cost than heading towards the opposite direction. Metrics without the symmetry property are called quasi-metrics (Definition \ref{def:metric}). The first example of a quasi-metric appeared in \cite{Hau} where the author discussed distance notions on the power set of a metric space. Other examples arise in order and domain theory (\cite{Mat,Gou,Sch,Sch_ext}) and in dynamics (\cite{DGKZT,CasGioZav,CDFGBT,Za_ent,Za_ent_action}). We refer to \cite{DezDez} for a collection of different distance notions and to \cite{KV94,FleLin} for a wide introduction to quasi-metric spaces.
	
	In \cite{KemKunOta}, the authors introduced and studied an appropriate notion of hyperconvexity in the realm of quasi-metric spaces, called $q$-hyperconvexity (or Isbell-convexity in the mentioned paper) moving from the early work contained in \cite{Sal}. Adapting the work of Dress (\cite{Dre}), they proved that every quasi-metric space has an essentially unique $q$-hyperconvex hull and provided an explicit construction. Moreover, they showed that the $q$-hyperconvex hull coincides with the injective hull constructed in the category of quasi-metric spaces and non-expansive maps. Moving from that paper, a lot of effort has been paid to adapt classical results for hyperconvex spaces to the asymmetric context. Among those, let us mention \cite{AgyHaiKun1,ConKunOta,AgyHaiKun,KunOta,OtaSab}. Motivated by the recent results and applications around the hyperconvex hull, this contribution fits in this line of research.
	
	In this paper, using ideas coming from coarse geometry and developed in \cite{Zav}, we introduce the notion of $\Sym$-rough isometries (Definition \ref{def:Sym_rough_isom}), an asymmetric version of rough isometries, and we study the stability of the $q$-hyperconvex hull along these equivalences. Namely, we investigate two connected questions:
	\begin{compactenum}[(i)]
	\item\label{q1} Are the $q$-hyperconvex hulls of $\Sym$-roughly isometric quasi metric spaces also $\Sym$-roughly isometric?
	\item\label{q2} What are the quasi-metric spaces that are $\Sym$-roughly isometric to $q$-hyperconvex metric spaces?
	\end{compactenum}
	We provide complete answers to both of them. As for \eqref{q1}, we extend the Gromov-Hausdorff distance to quasi-metric spaces (Definition \ref{def:GH_quasi-metric}) and show its connection with $\Sym$-rough isometries (Corollary \ref{coro:GH_sym_rough_iso}). By characterising the Gromov-Hausdorff distance using correspondences (Proposition \ref{prop:GH_correspondences}), we prove that it coincides with the more general network distance introduced and studied in \cite{CarMemRibSeg,ChoMem1,ChoMem2,ChoMem3}, which was used to prove the stability of some persistence diagrams (Remark \ref{rem:network_distance}). Let us also mention that a different notion of Gromov-Hausdorff distance for quasi-metric spaces was discussed in \cite{SheZha}. Furthermore, we prove that the Gromov-Hausdorff distance of two $q$-hyperconvex hulls is bounded from above by eight times the Gromov-Hausdorff distance of the two original quasi-metric spaces (Theorem \ref{theo:stability_q_hyperconvex}). 
	
	Regarding \eqref{q2}, we introduce $\Sym$-coarsely injective quasi-metric spaces (Definition \ref{def:Sym_coarsely_inj}), miming coarsely injective metric spaces. Then we prove that they are precisely those that are $\Sym$-large in their $q$-hyperconvex hull (Theorem \ref{theo:Sym_coarsely_injective}), and, moreover, those that are $\Sym$-roughly isometric to a $q$-hyperconvex quasi-metric space. 
	
	In Section \ref{sec:qms} we provide the needed background and recall the basic definitions of quasi-metric spaces and morphisms between them. In \S\ref{sub:hyper} we focus on hyperconvexity, describe the construction of the hyperconvex hull, and recall some known results. In \S\ref{sub:qhyper} we present $q$-hyperconvexity and the $q$-hyperconvex hull of a quasi-metric space. Section \ref{sec:Sym_rough_iso} is devoted to the introduction of $\Sym$-rough isometries and, in \S\ref{sub:GH}, of the extension of the Gromov-Hausdorff distance to quasi-metric spaces. In the last section, \S\ref{sec:main_results}, we collect the main results of this work, namely the stability of the $q$-hyperconvex hull (in \S\ref{sub:stability_qhyper}) and the definition and characterisation of $\Sym$-coarsely injective quasi-metric spaces (in \S\ref{sub:Sym_coarsely_inj}).

	\section{Quasi-metric spaces and $q$-hyperconvex hull}\label{sec:qms}
	
	In this paper, we write $\N$ for the set of natural numbers, including $0$, $\R$ for the real numbers, $\R_{\geq 0}=\{x\in\R\mid x\geq 0\}$ and $\R_{>0}=\{x\in\R\mid x>0\}$.
	
		For a set $X$ and a map $d\colon X\times X\to\R_{\geq 0}\cup\{\infty\}$, we consider the following properties:
		\begin{compactenum}[(M1)]
			\item\label{M1} for every $x,y\in X$, $d(x,y)=0$ and $d(y,x)=0$ if and only if $x=y$;
			\item\label{M1*} for every $x\in X$, $d(x,x)=0$;
			\item\label{M2} for every $x,y,z\in X$, $d(x,y)\leq d(x,z)+d(z,y)$;
			\item\label{M3} for every $x,y\in X$, $d(x,y)=d(y,x)$;
			\item\label{M4} for every $x,y\in X$, $d(x,y)<\infty$.
		\end{compactenum}
	
	\begin{definition}\label{def:metric}
		Let $X$ be a set and $d\colon X\times X\to\R_{\geq 0}\cup\{\infty\}$. Then $(X,d)$ and $d$ are called an {\em extended pseudo-quasi-metric space} and an {\em extended pseudo-quasi-metric}, respectively, if (M\ref{M1*}) and (M\ref{M2}) hold. Moreover, 
		\begin{compactenum}[(a)]
		\item the prefix {\em extended} is dropped if (M\ref{M4}) is fulfilled,
		\item the prefix {\em pseudo} is removed if (M\ref{M1}) holds (which implies (M\ref{M1*})), and
		\item the prefix {\em quasi} is cancelled if (M\ref{M3}) is satisfied.
		\end{compactenum}
	\end{definition}

For an extended (pseudo-)quasi-metric space $(X,d)$, we define the {\em conjugate extended (pseudo-)quasi-metric $d^{-1}$} as $d^{-1}(x,y)=d(y,x)$ for every $x,y\in X$, and its {\em symmetrisation} $d^s$ as the extended (pseudo-)metric $d^s=\max\{d,d^{-1}\}$, i.e., $d^s(x,y)=\max\{d(x,y),d(y,x)\}$ for every $x,y\in X$. Moreover, $d$ satisfies (M\ref{M4}) if and only if so does $d^{-1}$ if and only if so does $d^s$.

If $Y$ is a subset of an extended (pseudo)-quasi-metric space $(X,d)$, we denote by $d|_Y$ the restriction $d|_{Y\times Y}$ of the extended (pseudo-)quasi-metric to $Y$. Then $(Y,d|_Y)$ is again an extended (pseud)-quasi-metric space.

For an extended (pseudo-)quasi-metric space $(X,d)$, a point $x\in X$ and a radius $r\geq 0$, we denote by
$$B_d(x,r)=\{y\in X\mid d(x,y)<r\},\text{ and }C_d(x,r)=\{y\in X\mid d(x,y)\leq r\}.$$
If there is no risk of ambiguity, we drop the explicit mention of the quasi-metric $d$ in the notation. The family $\{B(x,r)\mid x\in X,r\in\R_{>0}\}$ forms a base the topology $\tau_d$ induced by $d$. Note that it is not true in general that $C_d(x,r)$ is closed in $\tau_d$ because of the lack of symmetry (see Example \ref{ex:quasi-metrics}). However, those balls are actually closed in $\tau_{d^{-1}}$. In the literature (e.g., in \cite{KemKunOta}), quasi-metrics are also called $T_0$-quasi-metric spaces because the induced topology satisfies the $T_0$ separation axiom. In those papers, they refer to a pseudo-quasi-metric as a quasi-metric. 

	A quasi-metric space $(X,d)$ is said to be {\em metrically convex} (\cite{KemKunOta}) if, for every $x,y\in X$ and every $r,s\in\R_{\geq 0}$ such that $d(x,y)\leq r+s$, there exists $z\in X$ such that $d(x,z)\leq r$ and $d(z,y)\leq s$ (equivalently, $z\in C_d(x,r)\cap C_{d^{-1}}(y,s)$). 
	
	A quasi-metric space $X$ is {\em bounded} if $\diam X=\sup\{d(x,y)\mid x,y\in X\}<\infty$.
	
	Before providing some examples of quasi-metric spaces, let us fix some notation borrowed from order theory. If $a,b\in\R$, we denote by $a\vee b=\max\{a,b\}$. We extend this notation also to maps: if $f,g\colon S\to\R$ are maps from a set $S$ to the real numbers, then $f\vee g\colon X\to\R$ is defined by $(f\vee g)(x)=f(x)\vee g(x)$ for every $x\in X$. Therefore, $d^s=d\vee d^{-1}$.
	\begin{example}\label{ex:quasi-metrics}
\begin{compactenum}[(a)]
	\item The map $u\colon\R\times\R\to\R_{\geq 0}$ associating to every pair of numbers $x,y\in\R$ the value $u(x,y)=x\dot{-}y=(x-y)\vee 0$ is a quasi-metric on $\R$. Moreover, for every $a\in\R$ and $r\geq 0$, $C_u(a,r)=[a-r,\infty)$, which is not closed since it is contained in every open ball centred in a point $x\in\R\setminus[a-r,\infty)$.
		\item The Sierpi\'nski pair is the set $\mathbb S=\{0,1\}$ with the quasi-metric $u|_{\mathbb S}$ that satisfies $u(0,1)=0$ and $u(1,0)=1$. Then $C_u(1,1/2)=\{1\}$, but $\{0\}\notin\tau_d$. 
	\end{compactenum}
	\end{example}
	
	\begin{definition}
		A map $\varphi\colon(X,d_X)\to(Y,d_Y)$ between quasi-metric spaces is called {\em roughly non-expansive} if there exists $\varepsilon\in\R_{\geq 0}$ such that $d_Y(\varphi(x),\varphi(y))\leq d_X(x,y)+\varepsilon$ for every $x,y\in X$. If we need to specify the constant $\varepsilon$, we call it {\em $\varepsilon$-roughly non-expansive}. A $0$-roughly non-expansive map is simply called {\em non-expansive}.
	\end{definition}

Recall that a map $\varphi\colon (X,d_X)\to (Y,d_Y)$ is
\begin{compactenum}[(a)]
\item an {\em isometric embedding} if, for every $x,x^\prime\in X$, $d_Y(\varphi(x),\varphi(x^\prime))=d_X(x,x^\prime)$, and
\item an {\em isometry} if it is a surjective isometric embedding.
\end{compactenum}
Equivalently, the map $\varphi$ is an isometry if and only if it is bijective and both $\varphi$ and $\varphi^{-1}$ are non-expansive.


\subsection{Hyperconvex hull of a metric space}\label{sub:hyper}
	
	Let us first recall the definition of hyperconvex space and hyperconvex hull.
		\begin{definition}{\rm (\cite[Definition 1]{AroPan})}\label{def:hyperconvex}
		A metric space $X$ is said to be {\em hyperconvex} if every family $\{(x_i,r_i)\}_{i\in I}\subseteq X\times\R_{\geq 0}$ with the property that $d(x_i,x_j)\leq r_i+r_j$ for every $i,j\in I$ satisfies $\bigcap_{i\in I}C(x_i,r_i)\neq\emptyset$.
	\end{definition}
Equivalently, as mentioned in the Introduction, a metric space $X$ is hyperconvex if and only if it is metrically convex and satisfies the {\em Helly property}, i.e., for every family $\{C(x_i,r_i)\}_{i\in I}$ of closed balls of $X$, $\bigcap_iC(x_i,r_i)\neq\emptyset$ provided that $C(x_i,r_i)\cap C(x_j,r_j)\neq\emptyset$ for every $i,j\in I$.
	
	A metric space $X$ is {\em injective in the category $\Met$} (i.e., the category of metric spaces and non-expansive maps between them) if, for every metric space $Y$ and every subspace $Z\subseteq Y$, a non-expansive map $f\colon Z\to X$ can be extended to a non-expansive map $\overline f\colon Y\to X$, i.e., $\overline f|_Z=f$.
	
	\begin{proposition}{\rm (\cite[Theorem 2, Theorem 3]{AroPan})}
		A metric space is hyperconvex if and only if it is injective in the category $\Met$.
	\end{proposition}
	Every metric space $X$ has its {\em hyperconvex hull} $E(X)$, which is the smallest (up to isometry) hyperconvex metric space that contains an isometric copy of $X$ as a subspace. An explicit construction can be found, for example, in \cite{Dre}. Let us briefly review it to better understand its counterpart for quasi-metric spaces, which will be described in \S\ref{sub:qhyper}.
	
	Let $X$ be a metric space. Define the family
	$$\Delta(X)=\{f\colon X\to\R_{\geq 0}\mid d(x,y)\leq f(x)+f(y)\},$$
	which can be endowed with the pointwise partial order $\leq$. More explicitly, for every $f,g\in\Delta(X)$, $f\leq g$ if $f(x)\leq g(x)$ for every $x\in X$. Let then $E(X)$ be the subset of minimal functions in the poset $(\Delta(X),\leq)$. We can endow $E(X)$ with the supremum norm $\lvert\lvert\cdot\rvert\rvert_\infty\colon f\to \sup\{\lvert f(x)\rvert\mid x\in X\}$.
	The extended metric induced by $\lvert\lvert\cdot\rvert\rvert_\infty$ on $E(X)$ is actually a metric. In the sequel, we simply write $\lvert\lvert\cdot\rvert\rvert$ instead of $\lvert\lvert\cdot\rvert\rvert_\infty$.
\begin{theorem}
The map $\mathfrak e\colon x\mapsto d(x,\cdot)$, for every $x\in X$, is an isometric embedding of $X$ into $E(X)$. Moreover, $E(X)$ is hyperconvex, and the hyperconvex hull of $X$.
\end{theorem}

Let us now briefly recall the Gromov-Hausdorff distance, and, in order to do that, let us start with the Hausdorff distance. Given a metric space $(X,d)$, on its power set we can introduce the {\em Hausdorff quasi-metric $d_H^q$} (\cite{Hau}) as the extended pseudo-quasi-metric defined as follows: for every $Y,Z\subseteq X$,
$$d_H^q(Y,Z)=\inf\{r\geq0\mid Y\subseteq C_d(Z,r)\}.$$
Then the extended pseudo-metric $d_H=(d_H^q)^s$ is called the {\em Hausdorff metric}.

If $X$ and $Y$ are two metric spaces and $r\geq 0$, the {\em Gromov-Hausdorff distance} between $X$ and $Y$ is at most $r$ (and we denote it by $d_{GH}(X,Y)\leq r$) if there exists a third metric space $Z$ and two isometric embeddings $i_X\colon X\to Z$ and $i_Y\colon Y\to Z$ such that $d_H(i_X(X),i_Y(Y))\leq r$.
\begin{theorem}[\cite{LanPavZus}]\label{theo:continuity_E(X)}
For every pair of metric spaces $X$ and $Y$,
$$d_{GH}(E(X),E(Y))\leq 2d_{GH}(X,Y).$$
\end{theorem}
The inequality stated in Theorem \ref{theo:continuity_E(X)} is a refinement of a result obtained in \cite{Moe}. Moreover, the authors provided in \cite[Example 3.4]{LanPavZus} two finite metric spaces $X$ and $Y$ for which $d_{GH}(E(X),E(Y))=2d_{GH}(X,Y)$.

Let us recall another characterisation of the Gromov-Hausdorff distance, that allows us to describe a consequence of Theorem \ref{theo:continuity_E(X)} in Remark \ref{rem:unstable_hyper_hull}.

	Let $\varphi,\psi\colon S\to(X,d)$ be two maps from a set to an extended metric space. We say that $\varphi$ and $\psi$ are {\em close}, and we denote it by $\varphi\sim\psi$ if there exists $\varepsilon\in\R_{\geq 0}$ such that $d(\varphi(x),\psi(x))\leq \varepsilon$ for every $x\in S$. 
	
	A subset $Y$ of an extended metric space $X$ is {\em large} if there exists a constant $\varepsilon\in\R_{\geq 0}$ such that $C(Y,\varepsilon)=X$. We say that it is {\em $\varepsilon$-large} if we want to emphasise the constant.
	\begin{definition}
		A map $\varphi\colon (X,d_X)\to (Y,d_Y)$ between metric spaces is a {\em rough isometry} if one of the following equivalent conditions hold:
		\begin{compactenum}[(a)]
			\item $\varphi$ is roughly non-expansive and there exists another rough non-expansive map $\psi\colon Y\to X$ such that $\varphi\circ\psi\sim id_Y$ and $\psi\circ\varphi\sim id_X$;
			\item there exists a constant $\varepsilon\in\R_{\geq 0}$ such that $\varphi(X)$ is $\varepsilon$-large in $Y$ and, for every $x,y\in X$,
			\begin{equation}\label{eq:rough_isometric_embedding}d_X(x,y)-\varepsilon\leq d_Y(\varphi(x),\varphi(y))\leq d_X(x,y)+\varepsilon.\end{equation}
		\end{compactenum}
		Two metric spaces are {\em roughly isometric} if there exists a rough isometry between them.
	\end{definition}
A $\varepsilon$-rough isometry is also known as a $\varepsilon$-isometry in the literature (we refer to \cite[Section IV.45]{Har} for a brief historical discussion).
	
\begin{remark}
	Let $\varphi,\psi\colon X\to Y$ be two close maps between metric spaces. If $\varphi$ is roughly non-expansive, then so is $\psi$. In particular, if $\varphi$ is non-expansive, then its entire closeness-equivalence class consists of roughly non-expansive maps. However, it is not true that for every roughly non-expansive map $\lambda\colon X\to Y$ one can find a non-expansive map $\xi\colon X\to Y$ that is close to $\lambda$. To show this, consider first that a non-expansive map is automatically continuous while a roughly non-expansive map does not respect the topology in general. For example, the floor map $\lambda\colon\R\to\Z$ (i.e., $\lambda\colon x\mapsto\max\{y\in\Z\mid y\leq x\}$) is roughly non-expansive, but is not close to any non-expansive map since the only continuous maps collapses the domain to a point.
\end{remark}

	\begin{proposition}{\rm (\cite[Corollary 7.3.28]{BurBurIva})}\label{prop:GH_and_rough_isom}
	Let $X$ and $Y$ be two metric spaces and $\varepsilon\in\R_{> 0}$. 
	\begin{compactenum}[(a)]
		\item If $d_{GH}(X,Y)<\varepsilon$, then $X$ and $Y$ are $\varepsilon$-roughly isometric.
		\item If $X$ and $Y$ are $\varepsilon$-roughly isometric, then $d_{GH}(X,Y)<2\varepsilon$.
	\end{compactenum}
\end{proposition}
We provide a proof of a more general result in Proposition \ref{prop:GH_and_rough_isom}.

\begin{remark}\label{rem:unstable_hyper_hull}
According to Proposition \ref{prop:GH_and_rough_isom}, Theorem \ref{theo:continuity_E(X)} implies that, 
\begin{equation}\label{eq:roug_iso_stability}\text{if two metric spaces are roughly isometric, then so are their hyperconvex hulls.}\end{equation} This result itself is non-trivial if we take into account that the hyperconvex hull is very unstable for other equivalence classes. Let us first recall that a map $\varphi\colon X\to Y$ between metric spaces is a {\em coarse equivalence} (see, for example, \cite{Roe}) if $\varphi(X)$ is large in $Y$ and there exist two monotonous maps $\rho_-,\rho_+\colon\R_{\geq 0}\to\R_{\geq 0}$ such that $\lim_{t\to\infty}\rho_-(t)=\infty$ and, for every $x,y\in X$,
\begin{equation}\label{eq:ce}\rho_-(d_X(x,y))\leq d_Y(\varphi(x),\varphi(y))\leq\rho_+(d_X(x,y)).\end{equation}
It immediately follows from the definition that, if $\varphi$ is bi-Lipschitz (i.e., there is $L>0$ such that $L^{-1}d_X(x,y)\leq d_Y(\varphi(x),\varphi(y))\leq Ld_X(x,y)$ for $x,y\in X$) and bijective, then it is a coarse equivalence. Another example of coarse equivalences is provided by rough isometries.

For every $n\in\N$, let us consider $\R^n$ endowed with the $\ell^1$-metric $d_1$ and the $\ell^\infty$-metric $d_\infty$. Clearly, $(\R^n,d_1)$ and $(\R^n,d_\infty)$ are bi-Lipschitz equivalent. According to \cite[Example 5.5]{Lan} and since $(\R^n,d_\infty)$ is hyperconvex, $E(\R^n,d_1)$ and $E(\R^n,d_\infty)$ are isometric to $(\R^{2^{n-1}},d_\infty)$ and $(\R^n,d_\infty)$, respectively. This observation has two consequences. First of all, $E(\R^n,d_1)$ is isometric to $E(\R^{2^{n-1}},d_\infty)$ since they are both isometric to $(\R^{2^{n-1}}d_\infty)$. However, $(\R^n,d_1)$ and $(\R^{2^{n-1}},d_\infty)$ are not coarsely equivalent. A possible proof of this result relies on the dimension notion of {\em asymptotic dimension}, which is a {\em coarse invariant} (i.e., two coarsely equivalent metric spaces share the same value of the asymptotic dimension) introduced by Gromov (\cite{Gro}). The result follows from the fact that the asymptotic dimension of $\R^m$ with any $\ell^p$ metric is $m$ ($\asdim\R^m=m$). Therefore, the implication \eqref{eq:roug_iso_stability} cannot be reversed.

Moreover, even though the pair of spaces $(\R^n,d_1)$ and $(\R^n,d_\infty)$ are bi-Lipschitz equivalent, $E(\R^n,d_1)$ and $E(\R^n,d_\infty)$ are not even coarsely equivalent for $n\geq 3$ (again,it can be formally shown using the asymptotic dimension).
\end{remark}

\begin{remark}\label{rem:coarse_injective}
In \cite{HaeHodPet}, the author introduced the following notion. A metric space $(X,d)$ is {\em coarsely injective} if there exists $\delta\in\R_{\geq 0}$ such that, for every family $\{(x_i,r_i)\}_{i\in I}$ satisfying $d(x_i,x_j)\leq r_i+r_j$ for every $i,j\in I$, the intersection $\bigcap_iC(x_i,r_i+\delta)$ is non-empty. Moreover, they characterised coarsely injective metric spaces as those metric spaces $X$ such that $\mathfrak e(X)$ is large in $E(X)$. That result justifies the name choice for the property.
\end{remark}

	\subsection{$q$-hyperconvexity and $q$-hyperconvex hull}\label{sub:qhyper}

	The following definition can be explicitly found in \cite{KemKunOta}, but it is essentially due to Salbany (\cite[Proposition 9]{Sal}).
		\begin{definition}
		A quasi-metric space $(X,d)$ is {\em $q$-hyperconvex} (or {\em Isbell-convex}) if, for each family $\{(x_i,r_i,s_i)\}_{i\in I}\subseteq X\times\R_{\geq 0}\times\R_{\geq 0}$, if $\bigcap_{i\in I}(C_d(x_i,r_i)\cap C_{d^{-1}}(x_i,s_i))\neq\emptyset$ provided that $d(x_i,x_j)\leq r_i+s_j$ for every $i,j\in I$.
	\end{definition}
	\begin{fact}\label{fact:qhyper_to_metrically_convex}{\rm (\cite[Proposition 1]{KemKunOta})}
	A $q$-hyperconvex quasi-metric space is metrically convex.
\end{fact}
\begin{proof}
Let $(X,d)$ be a $q$-hyperconvex space, $x,y\in X$ and $r,s\in\R_{\geq 0}$ such that $d(x,y)\leq r+s$. Define $r^\prime=s^\prime=d(y,x)$. Then we can apply the $q$-hyperconvexity to the family $\{(x,r,s^\prime),(y,r^\prime,s)\}$, and so there exists 
$$z\in C_d(x,r)\cap C_{d^{-1}}(x,s^\prime)\cap C_d(y,r^\prime)\cap C_{d^{-1}}(y,s)\subseteq C_d(x,r)\cap C_{d^{-1}}(y,s).$$
\end{proof}
	
	Let us warn the reader that a hyperconvex metric space is not necessarily $q$-hyperconvex, as shown in Example \ref{ex:q_hyper}. Before providing the promised example, let us state some general preservation properties.
	
	\begin{proposition}\label{prop:preservation_q_hyperconvexity}
	Let $(X,d_X)$ and $(Y,d_Y)$ be two $q$-hyperconvex quasi-metric spaces.
	\begin{compactenum}[(a)]
	\item {\rm (\cite[Proposition 2]{KemKunOta})} The quasi-metric space $(X,d^{-1}_X)$ is $q$-hyperconvex, and $(X,d^s_X)$ is hyperconvex.
	\item The product space $(X\times Y,d)$, where $d=d_X\vee d_Y$ is defined by $d((x,y),(x^\prime,y^\prime))=d_X(x,x^\prime)\vee d_Y(y,y^\prime)$ for every $(x,y),(x^\prime,y^\prime)\in X\times Y$, is $q$-hyperconvex.
	\end{compactenum}
	\end{proposition}
\begin{proof}
	(b) Let $\{((x_i,y_i),r_i,s_i)\}_{i\in I}\subseteq (X\times Y)\times\R_{\geq 0}\times\R_{\geq 0}$ be such that $d((x_i,y_i),(x_j,y_j))\leq r_i+s_j$ for every $i,j\in I$. Therefore, $d_X(x_i,x_j)\leq r_i+s_j$ and $d_Y(y_i,y_j)\leq r_i+s_j$, and so there are
	$$\overline x\in\bigcap_{i\in I}(C_{d_X}(x_i,r_i)\cap C_{d_X^{-1}}(x_i,s_i))\quad\text{and}\quad\overline y\in\bigcap_{i\in I}(C_{d_Y}(y_i,r_i)\cap C_{d_Y^{-1}}(y_i,s_i)).$$
	It is easy to show that
	$$(\overline x,\overline y)\in\bigcap_{i\in I}(C_d((x_i,y_i),r_i)\cap C_{d^{-1}}((x_i,y_i),s_i)).$$
\end{proof}
	
	\begin{example}\label{ex:q_hyper}
		\begin{compactenum}[(a)]
		\item The quasi-metric space $(\R,u)$ defined in Example \ref{ex:quasi-metrics}(a) is $q$-hyperconvex (\cite[Example 1]{KemKunOta}). Therefore, Proposition \ref{prop:preservation_q_hyperconvexity} implies that $(\R,u^{-1})$ is $q$-hyperconvex and $\R$ with the usual metric $d$ is hyperconvex since it coincides with $(\R,u^s)$. However, $(\R,d)$ is not $q$-hyperconvex as shown in \cite[Example 2]{KemKunOta}. Let us repeat the argument. For every $i\in[0,1]$, define $r_i=1/4$ and $s_i=3/4$. Then, for every $i,j\in[0,1]$, $d(i,j)\leq 1=r_i+s_j$. However,
		$$\bigcap_{i\in[0,1]}(C_d(i,r_i)\cap C_d(i,s_i))\subseteq C_d(0,1/4)\cap C_d(1,1/4)=\emptyset.$$
		
		Moreover, according to Proposition \ref{prop:preservation_q_hyperconvexity}(b), $(\R^n,u_n)$, where
		$$u_n((x_1,\dots,x_n),(y_1,\dots,y_n))=u(x_1,y_1)\vee\cdots\vee u(x_n,y_n)$$
		for every $(x_1,\dots,x_n),(y_1,\dots,y_n)\in\R^n$, is $q$-hyperconvex. So, also $(\R^n,u_n^{-1})$ is $q$-hyperconvex. Again, $(\R^n,u_n^s)$, which is $\R^n$ endowed with the metric $d_\infty$, is not $q$-hyperconvex, and the proof is analogous to that of $(\R,d)$.
		
		It is easy to see that, for every $n\in\N$, $(\R^n,d_\infty)$ is isometric to the diagonal of the quasi-metric space $(\R^n\times\R^n,u_n\vee u_n^{-1})$, where $u_n\vee u_n^{-1}$ is the quasi-metric described in Proposition \ref{prop:preservation_q_hyperconvexity}(b). Furthermore, $(\R^n\times\R^n,u_n\vee u_n^{-1})$ is $q$-hyperconvex again because of Proposition \ref{prop:preservation_q_hyperconvexity}(b). This idea can be found in \cite[Example 3]{KemKunOta} for $n=1$.
		\item Consider the Sierpi\'nski pair $(\mathbb S,u)$ as defined in Example \ref{ex:quasi-metrics}(b). It is easy to see that it is not metrically convex (and so it is not $q$-hyperconvex for Fact \ref{fact:qhyper_to_metrically_convex}). For example, $u(1,0)\leq 1/2+1/2$, but $C_u(1,1/2)\cap C_{u^{-1}}(1,1/2)=\emptyset$. According to \cite[Example 8]{KemKunOta}, $Q(\mathbb S,u)=([0,1],u)$. Taking advantage of the same ideas used in item (a), we notice that $(\mathbb S,u^s)=(\{0,1\},d)$, where $d$ is the usual metric, is not $q$-hyperconvex. Moreover, we can see that the injection $\varphi\colon(\{0,1\},d)\to([0,1]^2,u\vee u^{-1})$ defined by $0\mapsto (0,0)$ and $1\mapsto(1,1)$ is an isometric embedding and $([0,1]^2,u\vee u^{-1})$ is $q$-hyperconvex. Note that $([0,1]^2,u\vee u^{-1})$ is the $q$-hyperconvex hull of $(\{0,1\},d)$ as it is isomorphic to the $q$-hyperconvex hull provided in \cite[Example 7]{KemKunOta}.
		\end{compactenum}
	\end{example}
	
	Denote by $\QMet$ the category of quasi-metric spaces and non-expansive maps between them. Note that $\Met$ is a full subcategory of $\QMet$. Recall that a subcategory $\mathcal Y$ of a category $\mathcal X$, where $\Ifun\colon\mathcal Y\to\mathcal X$ is the inclusion functor, is said to be {\em full} if, for every pair of objects $X,Y\in\mathcal Y$, $\Mor_{\mathcal Y}(X,Y)$ is isomorphic to $\Mor_{\mathcal X}(\Ifun(X),\Ifun(Y))$.
	\begin{theorem}{\rm (\cite[Theorem 1]{KemKunOta}, \cite[Proposition 9]{Sal} with a slightly different terminology, and \cite[Proposition3.7]{Agy1} for a proof not relying on Zorn's Lemma)}\label{theo:q_hyper_iff_injective}
		A quasi-metric space $X$ is $q$-hyperconvex if and only if it is injective in the category $\QMet$, i.e., for every quasi-metric space $Y$, every subspace $Z$ of $Y$ and every non-expansive map $f\colon Z\to X$, $f$ can be extended to a non-expansive map $g\colon Y\to X$.
	\end{theorem}

Let us present the construction of the $q$-hyperconvex hull provided in \cite{KemKunOta}.

	Let $(X,d)$ be a quasi-metric space. A pair $f=(f_1,f_2)$ of maps $f_i\colon X\to\R_{\geq 0}$, for $i=1,2$, is {\em ample} if, for every $x,y\in X$,
$$d(x,y)\leq f_2(x)+f_1(y).$$
Let us denote by $P(X)$ the family of all ample pair of maps of $X$. We endow this set with a straightforward extension of the pointwise partial order, namely, $(f_1,f_2)\leq(g_1,g_2)$ if $f_i\leq g_i$ for $i=1,2$. The family $P(X)$ can be endowed with the extended quasi-metric $D$ defined as follows: for every $f=(f_1,f_2),g=(g_1,g_2)\in P(X)$,
$$D(f,g)=(\sup_{x\in X}(f_1(x)\dot{-}g_1(x)))\vee(\sup_{x\in X}(g_2(x)\dot{-}f_2(x))).$$
Note, moreover, that
$$D^s(f,g)=\lvert\lvert f_1-g_1\rvert\rvert\vee\lvert\lvert f_2-g_2\rvert\rvert.$$

The {\em $q$-hyperconvex hull} is the subset $Q(X)\subseteq P(X)$ of minimal ample pair of functions of $X$ endowed with $D|_{Q(X)}$, which is actually a quasi-metric (\cite[Remark 6]{KemKunOta}).
\begin{proposition}{\rm (\cite[Lemma 4, Proposition 7]{KemKunOta})}  Let $X$ be a quasi-metric space.
\begin{compactenum}[(a)]
\item The quasi-metric space $(Q(X),D)$ is $q$-hyperconvex. 
\item The map $\mathfrak q\colon X\to Q(X)$ associating to every $x\in X$ the pair of maps $\mathfrak q(x)=(d_{x,1},d_{x,2})$ where, for every $y\in X$, $d_{x,1}\colon y\mapsto d(x,y)$ and $d_{x,2}\colon y\mapsto d(y,x)$, is an isometric embedding.
\item No proper subspace of $Q(X)$ containing $\mathfrak q(X)$ is $q$-hyperconvex, and $Q(X)$ is the unique quasi-metric space having those properties up to isometry. Therefore, $Q(X)$ is the $q$-hyperconvex hull of $X$. 
\end{compactenum}
\end{proposition}

Let us collect some results concerning the $q$-hyperconvex hull that will be useful in the sequel.
\begin{proposition}\label{prop:Q(X)_basics}
	Let $X$ be a quasi-metric space.
	\begin{compactenum}[(a)]
		\item {\rm (\cite[Lemma 6]{KemKunOta})} Let $f=(f_1,f_2)\in P(X)$. Then $f\in Q(X)$ if and only if, for every $x\in X$,
		$$f_1(x)=\sup_{y\in X}(d(y,x)\dot{-}f_2(y)),\text{ and }f_2(x)=\sup_{y\in X}(d(x,y)\dot{-}f_1(y)).$$
		\item {\rm (\cite[Lemma 3]{KemKunOta})} For every $f\in Q(X)$ and $x\in X$, 
		$$f_1(x)-f_1(y)\leq d(y,x),\text{ and }f_2(x)-f_2(y)\leq d(x,y).$$
	\end{compactenum}
\end{proposition}

\begin{proposition}\label{prop:from_P(X)_to_Q(X)}{\rm (\cite[Proposition 2]{AgyHaiKun})}
	Let $X$ be a quasi-metric space. There exists a map $p\colon P(X)\to Q(X)$ satisfying the following conditions:
	\begin{compactenum}[(a)]
		\item $D(p(f),p(g))\leq D(f,g)$, for every $f,g\in P(X)$;
		\item $p(f)\leq f$, for every $f\in P(X)$, and so a map $g\in P(X)$ satisfies $g\in Q(X)$ if and only if $p(g)=g$.
	\end{compactenum}
\end{proposition}
The map $p$ of the previous proposition is obtained as a pointwise limit. We provide more details of its construction as they will be useful in the sequel. Let $f=(f_1,f_2)\in P(X)$. Define the pair $f^\ast=(f_1^\ast,f_2^\ast)$ as follows: for every $x\in X$,
\begin{equation}\label{eq:Lips}f_1^\ast(x)=\sup_{y\in X}(d(y,x)\dot{-}f_2(y)),\text{ and }f_2^\ast(x)=\sup_{y\in X}(d(x,y)\dot{-}f_1(y)).\end{equation}
Note that every element $f\in P(X)$ satisfies $f^\ast\leq f$ and $f\in Q(X)$ if and only if $f=f^\ast$. Moreover, we consider the map $q\colon P(X)\to P(X)$ defined as follows: for every $f=(f_1,f_2)\in P(X)$, $q(f)=((f_1+f_1^\ast)/2,(f_2+f_2^\ast)/2)$. Then $f^\ast\leq q(f)\leq f$. Then $p(f)$ is obtained as a pointwise limit of the pairs $\{q^n(f)\}_n$.

If we need to specify the quasi-metric space $X$ to avoid ambiguities, we write $q_X$ and $p_X$.

\begin{proposition}\label{prop:Q_of_subset}{\rm (\cite[Lemma 9]{KemKunOta})}
	Let $X$ be a quasi-metric space and $Y\subseteq X$. There exists an isometric embedding $\iota\colon Q(Y)\to Q(X)$ with the following additional property: for every $f\in Q(Y)$, $\iota(f)|_Y=f$.
\end{proposition}
Let us recall the construction of the map $\iota$ in Proposition \ref{prop:Q_of_subset} as it will be used to prove our result. The map $\iota$ consists of a composite $p_X\circ \sigma|_{Q(Y)}$, where $\sigma\colon P(Y)\to P(X)$ satisfies $\sigma(f)|_Y=f$.


Let us discuss the connection between the hyperconvex hull of a metric space and its $q$-hyperconvex hull, continuing Example \ref{ex:q_hyper}.
\begin{remark}\label{rem:q_hyper_metric_space}
	We have discussed in Example \ref{ex:q_hyper} that the hyperconvex hull $E(X)$ of a metric space $X$ is not necessarily $q$-hyperconvex. In fact, $E(\R,d)=(\R,d)$ since it is already hyperconvex, but it is not $q$-hyperconvex. However, there is a canonical isometric embedding $h$ of $E(X)$ into $Q(X)$, and the embedding works as follows: for every $f\in E(X)$, take $h(f)=(f,f)\in Q(X)$ (\cite[Proposition 5]{KemKunOta}). This observation implies that $Q(X)$ is isometric to $Q(E(X))$. More precisely, $E(X)$ is isometric to the largest metric subspace of $Q(X)$ containing $X$ (\cite{Wil}).
\end{remark}	

\section{$\Sym$-rough isometries and Gromov-Hausdorff distance}\label{sec:Sym_rough_iso}

Let us consider the functor $\Sym\colon\RQMet\to\RMet$ associating to every object $(X,d)\in\RQMet$ its symmetrisation $(X,d^s)$ and keeping the morphisms unaltered. As in \cite{Zav}, this functor can be used to transport the closeness relation into the realm of (extended) quasi-metric spaces. A pair of parallel maps $f,g\colon S\to(X,d)$ from a set to a quasi-metric space is {\em $\Sym$-close} if $f$ and $g$ are close provided that $X$ is equipped with the metric $d^s$, and we write $f\sim_{\Sym}g$. Similarly, a subset $Y$ of a quasi-metric space $(X,d)$ is {\em $\Sym$-large} in $X$ if $Y$ is large in $(X,d^s)$. If we want to specify the constant $\varepsilon$, we say that $f$ and $g$ are {\em $\varepsilon$-$\Sym$-close} ($f\sim_{\Sym,\varepsilon}g$) and $Y$ is {\em $\varepsilon$-$\Sym$-large} in $X$.

\begin{definition}[Definition/Proposition]\label{def:Sym_rough_isom}
	A map $\varphi\colon(X,d_X)\to(Y,d_Y)$ between extended quasi-metric spaces is a {\em $\Sym$-rough isometry} if one of the following equivalent conditions holds:
	\begin{compactenum}[(a)]
		\item there exists a constant $\varepsilon\in\R_{\geq 0}$ such that $\varphi$ is $\varepsilon$-roughly non-expansive and there exists another $\varepsilon$-roughly non-expansive map $\psi\colon Y\to X$ (called {\em $\Sym$-rough inverse}) such that $\varphi\circ\psi\sim_{\Sym,\varepsilon}id_Y$ and $\psi\circ\varphi\sim_{\Sym,\varepsilon}id_X$;
		\item there exists a constant $\varepsilon\in\R_{\geq 0}$ such that $\varphi(X)$ is $\varepsilon$-$\Sym$-large in $Y$ and, for every $x,y\in X$, \eqref{eq:rough_isometric_embedding} holds.
	\end{compactenum}
\end{definition} 
\begin{proof}
	(a)$\to$(b) Let $\varepsilon\in\R_{\geq 0}$ be a constant for which $\varphi$ and $\psi$ are $\varepsilon$-roughly non-expansive, and, for every $x\in X$ and $y\in Y$, $d_X^s(x,\psi(\varphi(x)))\leq \varepsilon$ and $d_Y^s(y,\varphi(\psi(y)))\leq \varepsilon$. Thus, in particular, $\psi(X)$ is $\Sym$-roughly non-expansive. It remains to prove that $d_Y(\varphi(x),\varphi(x^\prime))\geq d_X(x,x^\prime)-\gamma$ for some $\gamma\in\R_{\geq 0}$. If $x,x^\prime\in X$, we have that
	$$d_X(x,x^\prime)\leq d_X(x,\psi(\varphi(x)))+d_X(\psi(\varphi(x)),\psi(\varphi(x^\prime)))+d_X(\psi(\varphi(x^\prime)),x^\prime)\leq d_Y(\varphi(x),\varphi(x^\prime))+3\varepsilon.$$
	
	(b)$\to$(a) For every $y\in Y$, there exists $x_y\in X$ such that $d^s_Y(y,\varphi(x_y))\leq \varepsilon$. Then we define $\psi\colon y\mapsto x_y$ for every $y\in Y$. We need to prove that $\psi$ is $\Sym$-roughly non-expansive. If $y,y^\prime\in Y$,
	$$\begin{aligned}d_X(\psi(y),\psi(y^\prime))&\,=d_X(x_y,x_{y^\prime})\leq d_Y(\varphi(x_y),\varphi(x_{y^\prime})+\varepsilon\leq\\
		&\,\leq d_Y(\varphi(x_y),y)+d_Y(y,y^\prime)+d_Y(y^\prime,\varphi(x_{y^\prime}))+\varepsilon\leq d_Y(y,y^\prime)+3\varepsilon.\end{aligned}$$
	Finally, $\varphi\circ\psi\sim_{\Sym}id_Y$ by definition of $\psi$, and, for every $x\in X$,
	$$d_X^s(x,\psi(\varphi(x)))=d_X^s(x,x_{\varphi(x)})\leq d_Y^s(\varphi(x),\varphi(x_{\varphi(x)}))+\varepsilon\leq 2\varepsilon.$$
\end{proof}
We refer to a map $\varphi\colon X\to Y$ as an {\em $\varepsilon$-$\Sym$-rough isometry} for some $\varepsilon\in\R_{\geq 0}$ if it fulfils Definition \ref{def:Sym_rough_isom}(b).


\begin{remark}\label{rem:X^s_roughly_iso}
Let $X$ and $Y$ be two quasi-metric spaces. It is easy to see that, if $X$ and $Y$ are $\Sym$-roughly isometric, then $\Sym X$ and $\Sym Y$ are roughly isometric. The opposite implication is not true, and we now provide a quasi-metric space $X$ such that $\Sym X$ is not $\Sym$-roughly isometric to $X$.

First of all, let us note the following fact. Suppose that $\varphi\colon Z\to W$ is a $\varepsilon$-$\Sym$-rough isometry from a quasi-metric space $Z$ to a metric space $W$. Then, for every $z,z^\prime\in Z$,
$$\lvert d_Z(z,z^\prime)-d_Z(z^\prime,z)\rvert\leq\lvert d_W(\varphi(z),\varphi(z^\prime))-d_W(\varphi(z^\prime),\varphi(z))\rvert+2\varepsilon=2\varepsilon.$$
Consider now $X=(\R,u)$ as in Example \ref{ex:quasi-metrics}. For every $x\in\R_{\geq 0}$, $\lvert u(0,x)-u(x,0)\rvert=x$, and so $X$ is not $\Sym$-roughly isometric to $(\R,d)=\Sym X$.
\end{remark}


\subsection{Gromov-Hausdorff distance between quasi-metric spaces}\label{sub:GH}

We want to extend the notion of Gromov-Hausdorff distance to the realm of quasi-metric spaces, and connect it to $\Sym$-rough isometries.

Given a quasi-metric space $Z$, we define its {\em $\Sym$-Hausdorff distance $d_H^{\Sym}$} on $\mathcal P(Z)$ as the Hausdorff metric induced by $\Sym(Z)$.

\begin{definition}\label{def:GH_quasi-metric}
Let $X$ and $Y$ be two quasi-metric spaces and $\varepsilon>0$. The {\em $\Sym$-Gromov-Hausdorff distance} between $X$ and $Y$ is less than $\varepsilon$, and we write $d_{GH}^{\Sym}(X,Y)<\varepsilon$, if there exists a quasi-metric space $Z$ and two isometric embeddings $i_X\colon X\to Z$ and $i_Y\colon Y\to Z$ such that $d_H^{\Sym}(i_X(X),i_Y(Y))<\varepsilon$.
\end{definition}

\begin{proposition}\label{prop:SymGH=GH}
If $X$ and $Y$ are two metric spaces, then $d_{GH}(X,Y)=d_{GH}^{\Sym}(X,Y)$.
\end{proposition}
\begin{proof}
Clearly, $d_{GH}^{\Sym}(X,Y)\leq d_{GH}(X,Y)$. Suppose now that $d_{GH}^{\Sym}(X,Y)<\varepsilon$ for some $\varepsilon>0$, and let $i_X\colon X\to (Z,d)$ and $i_Y\colon Y\to(Z,d)$ be two isometric embeddings into a quasi-metric space such that $d_H^{\Sym}(X,Y)<\varepsilon$. Since $X$ and $Y$ are metric spaces, $i_X^\prime\colon X\to(Z,d^s)$ and $i_Y^\prime\colon Y\to(Z,d^s)$ are still isometric embeddings into a metric space, and $d_H(i_X^\prime(X),i_Y^\prime(Y))=d^{\Sym}_H(i_X^\prime(X),i_Y^\prime(Y))<\varepsilon$.
\end{proof}
By virtue of Proposition \ref{prop:SymGH=GH}, we refer to the $\Sym$-Gromov-Hausdorff distance simply as Gromov-Hausdorff distance and we denote it by $d_{GH}$.

Similarly to the classical, metric setting (for which we refer to \cite{BurBurIva}), we provide a characterisation of the Gromov-Hausdorff using correspondences (Proposition \ref{prop:GH_correspondences}) and $\Sym$-rough isometries (Corollary \ref{coro:GH_sym_rough_iso}). The proofs are adaptations of their metric counterparts.

Let $(X,d_X)$ and $(Y,d_Y)$ be two quasi-metric spaces. A relation $\mathcal R\subseteq X\times Y$ is a {\em correspondence} if $\pi_X(\mathcal R)=X$ and $\pi_Y(\mathcal R)=Y$, where $\pi_X$ and $\pi_Y$ are the canonical projections of the product $X\times Y$. We define  the {\em distorsion of the correspondence $\mathcal R$} as the value
$$\dis\mathcal R=\sup\{\lvert d_X(x,x^\prime)-d_Y(y,y^\prime)\rvert\mid(x,y),(x^\prime,y^\prime)\in\mathcal R\}.$$

\begin{proposition}\label{prop:GH_correspondences}
Let $X$ and $Y$ be two quasi-metric spaces. Then
$$d_{GH}(X,Y)=\frac{1}{2}\inf\{\dis\mathcal R\mid\mathcal R\subseteq X\times Y\text{ is a correspondence}\}.$$
\end{proposition}
\begin{proof}
Suppose that $d_{GH}(X,Y)<\varepsilon$, for some $\varepsilon>0$. We claim that there exists a correspondence $\mathcal R$ with $\dis\mathcal R<2\varepsilon$. Without loss of generality, there is a quasi-metric space $(Z,d)$ containing $X$ and $Y$ as subspaces and such that $d_H^{\Sym}(X,Y)<\varepsilon$. Define 
$$\mathcal R=\{(x,y)\in X\times Y\mid d^s(x,y)<\varepsilon\},$$
which is a correspondence. Moreover, for every $(x,x^\prime),(y,y^\prime)\in\mathcal R$,
$$\begin{gathered}d(x,x^\prime)-d(y,y^\prime)\leq (d(x,y)+d(y,x^\prime))-(d(y,x^\prime)-d(y^\prime,x^\prime))=
	 d(x,y)-d(y^\prime,x^\prime)<2\varepsilon \\
	 \text{ and, similarly, }\quad  d(y,y^\prime)-d(x,x^\prime)<2\varepsilon.\end{gathered}$$
Therefore, $\dis\mathcal R<2\varepsilon$.

Conversely, assume that there is a correspondence $\mathcal R$ with $\dis\mathcal R=2\varepsilon$. On the disjoint union $Z=X\sqcup Y$, define a quasi-metric $d$ as follows: $d|_X=d_X$, $d|_Y=d_Y$, and, if $x\in X$ and $y\in Y$,
$$\begin{gathered}d(x,y)=\inf\{d_X(x,x^\prime)+\varepsilon+d_Y(y^\prime,y)\mid(x^\prime,y^\prime)\in\mathcal R\}\text{ and}\\
d(y,x)=\inf\{d_Y(y,y^\prime)+\varepsilon+d_X(x^\prime,x)\mid (x^\prime,y^\prime)\in\mathcal R\}.\end{gathered}$$
Then $d$ is actually a quasi-metric thanks to $\dis\mathcal R=2\varepsilon$, and $d_H^{\Sym}(X,Y)=\varepsilon$ since $\mathcal R$ is a correspondence.
\end{proof}

Thanks to Proposition \ref{prop:GH_correspondences}, we can connect the Gromov-Hausdorff distance with another distance notion introduced in topological data analysis.

\begin{remark}\label{rem:network_distance}
According to \cite{ChoMem3}, a {\em network} is a set $X$ endowed with a function $w\colon X\times X\to\R$ with no further properties. In particular, $w$ can assume negative values and $w(x,x)$ can be non-$0$. Because of their flexibility, these structures can be used to represent real-life data coming from several situations. Examples can be found in social science, commerce and economy, neuroscience, biology and defence. We refer to \cite{ChoMem3} for further details and bibliography. 

Let $(X,w_X)$ and $(Y,w_Y)$ be two given networks. The notions of correspondence and distorsion can be straightforwardly extended to networks. Then the {\em network distance} of $(X,w_X)$ and $(Y,w_Y)$ is defined as follows: 
$$d_{\mathcal N}((X,w_X),(Y,w_Y))=\frac{1}{2}\inf\{\dis\mathcal R\mid \mathcal R\subseteq X\times Y\text{ is a correspondence}\}.$$
The network distance is actually an extended pseudo-metric on the space of networks. According to Proposition \ref{prop:GH_correspondences}, on the family of quasi-metric spaces, the network distance coincides with the Gromov-Hausdorff distance.

The authors of \cite{ChoMem3} used the network distance to prove two important stability results. Given a network $X$, two filtrations of simplicial complexes, the {\em Rips filtration} and the {\em Dowker filtration} (more precisely, both the {\em Dowker sink} and the {\em Dowker source filtrations}) can be constructed. To a filtration of simplicial complexes, we can associate a persistence diagram encoding the persistent homology of the filtration (we refer to \cite{EdeHar} for a general introduction to persistent homology and persistence diagrams as tools to investigate data sets). Persistence diagrams can be compared using a suitable metric, called the bottleneck distance (see, again, \cite{EdeHar}). If $X$ and $Y$ are two networks, the bottleneck distance between the persistence diagrams induced by the Rips filtrations is bounded by two times the network distance of $X$ and $Y$ (\cite[Proposition 12]{ChoMem3}), and a similar estimate can be proved for the Dowker filtrations (\cite[Proposition 15]{ChoMem3}).
\end{remark}

\begin{corollary}\label{coro:GH_sym_rough_iso}
Let $X$ and $Y$ be two quasi-metric spaces and $\varepsilon>0$.
\begin{compactenum}[(a)]
\item If $d_{GH}(X,Y)<\varepsilon$, then there exists a $\varepsilon$-$\Sym$-rough isometry $\varphi\colon X\to Y$.
\item If there exists a $\varepsilon$-$\Sym$-rough isometry $\varphi\colon X\to Y$, then $d_{GH}(X,Y)<2\varepsilon$.
\end{compactenum}
\end{corollary}
\begin{proof}
(a) Let $\mathcal R$ be a correspondence having $\dis\mathcal R<\varepsilon$. Since $\pi_X(\mathcal R)=X$, for every point $x\in X$ we can pick an arbitrary point $y_x\in Y$ satisfying $(x,y_x)\in\mathcal R$. We then define $\varphi\colon X\to Y$ by associating every $x\in X$ to the chosen point $\varphi(x)=y_x$. Since $\dis\mathcal R<\varepsilon$, for every $x,x^\prime\in X$,
$$\lvert d_Y(\varphi(x),\varphi(x^\prime))-d_X(x,x^\prime)\rvert<\varepsilon.$$
Moreover, for every $y\in Y$, there exists $x\in X$ with $(x,y)\in\mathcal R$. Furthermore, $(x,\varphi(x))\in\mathcal R$, and thus
$$d_Y(y,\varphi(x))=\lvert d_Y(y,\varphi(x))-d_X(x,x)\rvert<\varepsilon,\text{ and, similarly, }d_Y(\varphi(x),y)<\varepsilon.$$

(b) If $\varphi\colon X\to Y$ is a $\Sym$-rough isometry, define
$$\mathcal R=\{(x,y)\in X\times Y\mid d^s(\varphi(x),y)\leq\varepsilon\}.$$
Since $(x,\varphi(x))\in\mathcal R$ for every $x\in X$, and $\varphi(X)$ is $\varepsilon$-$\Sym$-large in $Y$, $\mathcal R$ is a correspondence. Moreover, if $(x,y),(x^\prime,y^\prime)\in\mathcal R$,
$$\lvert d_Y(y,y^\prime)-d_X(x,x^\prime)\rvert\leq\lvert d_Y(\varphi(x),\varphi(x^\prime))-d_X(x,x^\prime)\rvert+d_Y^s(\varphi(x^\prime),y^\prime)+d_Y^s(y,\varphi(x))\leq3\varepsilon.$$
Proposition \ref{prop:GH_correspondences} finally implies that $d_{GH}(X,Y)\leq \dis\mathcal R/2<2\varepsilon$.
\end{proof}

Note that Definition \ref{def:GH_quasi-metric} can be straightforwardly generalised to extended quasi-metric spaces, and counterparts of Propositions \ref{prop:GH_correspondences} and \ref{prop:GH_and_rough_isom} can be similarly deduced. For the sake of simplicity, and since we are going to apply the Gromov-Hausdorff distance only to quasi-metric spaces, we provided the above results in their weaker form.

Clearly, if $X$ and $Y$ are two isometric quasi-metric spaces, then $d_{GH}(X,Y)=0$. As for the converse implication, the following result extends the classical characterisation of metric spaces having Gromov-Hausdorff distance equal to $0$ (\cite[Theorem 7.3.30]{BurBurIva}).

\begin{theorem}
Let $X$ and $Y$ be two quasi-metric spaces such that $\Sym X$ and $\Sym Y$ are compact. Then $d_{GH}(X,Y)=0$ if and only if $X$ and $Y$ are isometric.
\end{theorem}
\begin{proof}
Assume that $\Sym X$ and $\Sym Y$ are compact and $d_{GH}(X,Y)=0$. Corollary \ref{coro:GH_sym_rough_iso} implies that, for every $n\in\N$, there is a $1/n$-$\Sym$-rough isometry $\varphi_n\colon X\to Y$. Since $\Sym X$ is compact, there exists a countable subset $Z$ of $X$ which is dense in $\Sym X$. Using the Cantor diagonal procedure, we can find a subsequence $\{\varphi_{n_k}\}_k$ of maps such that, for every $z\in Z$, $\{\varphi_{n_k}(z)\}_k$ converges in $\Sym Y$. For the sake of simplicity, and without loss of generality, we can assume that already the sequence $\{\varphi_n\}_n$ has the desired property. Then we define a map $\varphi\colon Z\to Y$ such that, for every $z\in Z$, $\varphi(z)$ is the limit of $\{\varphi_n(z)\}_n$ in $\Sym Y$. Moreover, for every $n\in\N$ and every $z,z^\prime\in Z$,
$$\lvert d_Y(\varphi(z),\varphi(z^\prime))-d_X(z,z^\prime)\rvert\leq d_Y^s(\varphi(z),\varphi_n(z))+d_Y^s(\varphi_n(z^\prime),\varphi(z^\prime))+\lvert d_Y(\varphi_n(z),\varphi_n(z^\prime))-d_X(z,z^\prime)\rvert,$$
and the right hand sided member converges to $0$. Therefore, $\varphi$ is an isometric embedding of $Z$ into $Y$.

We now want to extend $\varphi$ to a map $\varphi^\prime\colon X\to Y$. For every $x\in X$, there exists a sequence of points $\{z_n\}_n$ in $Z$ converging to $x$ in $\Sym X$. Define $\varphi^\prime(z)$ as the limit of the sequence $\{\varphi(z_n)\}_n$ in $\Sym Y$. This limit exists. In fact, $\{\varphi(z_n)\}_n$ is a Cauchy sequence because $\varphi$, and so $\Sym\varphi$, are isometric embeddings, and $\Sym Y$ is in particular complete. Let us now show that $\varphi^\prime$ is an isometric embedding. For every $x,x^\prime\in X$, fix two sequences $\{z_n\}_n$ and $\{z_n^\prime\}_n$ of points in $Z$ converging to $x$ and $x^\prime$, respectively, in $\Sym X$. Then, for every $n\in\N$,
$$\begin{aligned}\lvert d_Y(\varphi^\prime(x),\varphi^\prime(x^\prime))-d_X(x,x^\prime)\rvert\leq&\, d_Y^s(\varphi^\prime(x),\varphi(z_n))+d_Y^s(\varphi(z_n^\prime),\varphi^\prime(x^\prime))+d_X^s(x,z_n)+d_X^s(z_n^\prime,x^\prime)+\\
	&\quad\quad\quad\quad\quad\quad\quad\quad\quad\quad\quad\quad\quad\quad\quad+\lvert d_Y(\varphi(z_n),\varphi(z_n^\prime))-d_X(z_n,z_n^\prime)\rvert,\end{aligned}$$
and the latter summand converges to $0$.

With a similar technique, we can construct an isometric embedding $\psi\colon Y\to X$. In particular, both $\Sym\varphi$ and $\Sym\psi$ are isometric embeddings, and so is $\Sym\psi\circ\Sym\varphi$. However, $\Sym X$ is compact, and so \cite[Theorem 1.6.14]{BurBurIva} implies that $X=\Sym\psi(\Sym\varphi(X))=\psi(\varphi(X))$ and so $\varphi$ is bijective.
\end{proof}	

\begin{example}
There are quasi-metric spaces that are compact, even though their symmetrisations are not. Consider, for example, $X=(\R_{\geq 0},u|_{\R_{\geq 0}})$ where $u$ is defined as in Example \ref{ex:quasi-metrics}. In an open cover of $X$, it is enough to take a subcover consisting of just an open containing $0$ since that subset has to coincide with $X$ itself. However, $\Sym X$ is $\R_{\geq 0}$ with the usual euclidean metric, which is not compact.
\end{example}


\section{Main results}\label{sec:main_results}

	\subsection{Stability of the $q$-hyperconvex hull}\label{sub:stability_qhyper}
	
We now want to prove the following stability result, which can be seen as the counterpart of \cite[Theorem 1.55]{Moe} stating the same inequality for the hyperconvex hull of a metric space.
\begin{theorem}\label{theo:stability_q_hyperconvex}
Let $X$ and $Y$ be two quasi-metric spaces. Then 
$$d_{GH}(Q(X),Q(Y))\leq 8d_{GH}(X,Y).$$
\end{theorem}
This result will descend from the following crucial lemma.

	\begin{lemma}\label{prop:Sym_large}
		Let $X$ be a quasi-metric space and $Y\subseteq X$ be a subset. If $Y$ is $\varepsilon$-$\Sym$-large in $X$, then $\iota(Q(Y))$, defined as in Proposition \ref{prop:Q_of_subset}, is $4\varepsilon$-$\Sym$-large in $Q(X)$.
	\end{lemma}
\begin{proof}
Since $Y$ is $\varepsilon$-$\Sym$-large in $X$, for every $x\in X$, we can pick a point $y_x\in Y$ such that $d^s(x,y_x)\leq \varepsilon$.
	\begin{claim}\label{claim:Sym_large}
		If $f\in Q(X)$ and $g\in P(Y)$ satisfies $g\leq f|_Y$, then $g^\ast\geq f|_Y-2\varepsilon$.
	\end{claim}
	\begin{proof}[Proof of Claim \ref{claim:Sym_large}]
		Assume that there exists $\overline y\in Y$ such that $g^\ast_1(\overline y)<f_1(\overline y)-2\varepsilon$ (the case $g_2^\ast(\overline y)<f_2(\overline y)-2\varepsilon$ can be similarly treated). Since $f\in Q(X)$ and
		$$\begin{aligned}d(x,\overline y)-f_2(x)&\,\leq d(x,y_x)+d(y_x,\overline y)-f_2(y_x)+f_2(y_x)-f_2(x)\leq\\
			&\,\leq 2d^s(x,y_x)+d(y_x.\overline y)-f_2(y_x)\leq 2\varepsilon+d(y_x,\overline y)-f_2(y_x)\end{aligned}$$
		thanks to Proposition \ref{prop:Q(X)_basics}(b), we achieve the following contradiction:
		$$f_1(\overline y)=\sup_{x\in X}(d(x,\overline y)\dot{-}f_2(x))\leq 2\varepsilon+\sup_{y\in Y}(d(y,\overline y)\dot{-}g_2(y))=2\varepsilon+g^\ast_1(\overline y)<f_1(\overline y).$$
	\end{proof}
	
	
		Let $f\in Q(X)$. Claim \ref{claim:Sym_large} implies the following inequality.
		\begin{claim}\label{claim:D_S}
		$D^s_Y(p_Y(f|_Y),f|_Y)\leq 2\varepsilon$.
	\end{claim}
\begin{proof}[Proof of Claim \ref{claim:D_S}]
	By applying the result to the map $g=f|_Y$, we obtain that
		$$q_Y(f|_Y)=(g+g^\ast)/2\geq g^\ast\geq f|_Y-2\varepsilon$$
		since $g^\ast\leq g$. Moreover, $q_Y(f|_Y))\leq f|_Y$, and so we can reapply the same result to obtain that $q_Y^2(f|_Y)\geq f|_Y-2\varepsilon$. By induction and passing to the limit, we obtain that $p_Y(f|_Y)\geq f|_Y-2\varepsilon$. Therefore, since $p_Y(f|_Y)\leq f|_Y$, the claim follows.
		\end{proof}
		
		Let now $i\in\{1,2\}$, and $x\in X$. Using Proposition \ref{prop:Q(X)_basics}(b), we have that
		\begin{gather*}\begin{aligned}f_i(x)-\iota(p_Y(f|_Y))_i(x)&\,\leq f_i(x)-f_i(y_x)+f_i(y_x)-\iota(p_Y(f|_Y))_i(y_x)+\iota(p_Y(f|_Y))_i(y_x)-\iota(p_Y(f|_Y))_i(x)\leq\\
				&\,\leq 2\varepsilon+f_i(y_x)-\iota(p_Y(f|_Y))_i(y_x),\quad\text{and, similarly,}\end{aligned}\\ \iota(p_Y(f|_Y))_i(x)-f_i(x)\leq 2\varepsilon+\iota(p_Y(f|_Y))_i(y_x)-f_i(y_x).\end{gather*}
		Thus, Proposition \ref{prop:Q_of_subset} implies
		$$D_X^s(f,\iota(p_Y(f|_Y)))\leq 2\varepsilon+D_Y^s(f|_Y,\iota(p_Y(f|_Y))|_Y)=2\varepsilon+D^s_Y(f|_Y,p_Y(f|_Y))\leq 4\varepsilon.$$
	\end{proof}

\begin{proof}[Proof of Theorem \ref{theo:stability_q_hyperconvex}]
Suppose that $d_{GH}(X,Y)\leq \varepsilon$. Then, without loss of generality, there exists a quasi-metric space $Z$ containing $X$ and $Y$ as subspaces with the property that $d_H^{\Sym}(X,Y)\leq \varepsilon$. Therefore, both $X$ and $Y$ are $\varepsilon$-$\Sym$-large in $X\cup Y$. Thus Lemma \ref{prop:Sym_large} implies that
$$d_{GH}(Q(X),Q(Y))\leq d_{GH}(Q(X),Q(X\cup Y))+d_{GH}(Q(X\cup Y),Q(Y))\leq 8\varepsilon.$$ 
\end{proof}

In \cite{LanPavZus}, the authors showed that, for every pair of metric spaces $X$ and $Y$, $d_{GH}(E(X),E(Y))\leq 2d_{GH}(X,Y)$ refining \cite[Theorem 1.55]{Moe}. Moreover, they provided two bounded metric spaces $X$ and $Y$ such that $d_{GH}(E(X),E(Y))=2d_{GH}(X,Y)$. This observation motivates the following question.
\begin{question}
Can the inequality provided by Theorem \ref{theo:stability_q_hyperconvex} be refined, or there exists a pair of quasi-metric spaces $X$ and $Y$ such that $d_{GH}(Q(X),Q(Y))=8d_{GH}(X,Y)$?
\end{question}	

\begin{remark}
Similarly to what we have done for Theorem \ref{theo:continuity_E(X)}, Theorem \ref{theo:stability_q_hyperconvex} implies that the $q$-hyperconvex hulls of two quasi-metric spaces are $\Sym$-roughly equivalent if so are the original spaces. As in Remark \ref{rem:unstable_hyper_hull}, we want to show that the $q$-hyperconvex hull is unstable for other equivalence relations. Let us first recall an asymmetric counterpart of coarse equivalences, as introduced in \cite{Zav}. A map $\varphi\colon(X,d_X)\to(Y,d_Y)$ between quasi-metric spaces is a {\em $\Sym$-coarse equivalence} if $\varphi(X)$ is $\Sym$-large in $Y$ and there exist two monotonous maps $\rho_-,\rho_+\colon\R_{\geq 0}\to\R_{\geq 0}$ such that $\lim_{t\to\infty}\rho_-(t)=\infty$ and satisfying \eqref{eq:ce}. It is easy to see that, if $\varphi\colon X\to Y$ is a $\Sym$-coarse equivalence between quasi-metric spaces, then $\Sym\varphi\colon\Sym X\to\Sym Y$ is a coarse equivalence.

Building up on Remark \ref{rem:unstable_hyper_hull}, consider the metric spaces $X=(\R^n,d_\infty)$ and $Y=(\R^n,d_1)$, which are bi-Lipschitz equivalent. According to Remark \ref{rem:q_hyper_metric_space} and Example \ref{ex:q_hyper}(a),
$$E(X)=\R^n\subseteq Q(X)\subseteq(\R^n\times\R^n,u_n\vee u_n^{-1}),\quad\text{and}\quad E(Y)=\R^{2^{n-1}}\subseteq Q(Y)\subseteq(\R^{2^{n-1}}\times\R^{2^{n-1}},u_{2^{n-1}}\vee u_{2^{n-1}}^{-1}).$$
Therefore, 
$$\R^n\subseteq\Sym Q(X)\subseteq\R^{2n}\quad\text{and}\quad \R^{2^{n-1}}\subseteq\Sym Q(Y)\subseteq \R^{2^n},$$
and so, since the asymptotic dimension is monotone (\cite{Gro}),
$$n\leq\asdim \Sym Q(X)\leq 2n,\quad\text{and}\quad 2^{n-1}\leq\asdim\Sym Q(Y)\leq 2^n.$$
Thus $Q(X)$ and $Q(Y)$ are not $\Sym$-coarsely equivalent for $n\geq 5$.
\end{remark}
	
	\subsection{$\Sym$-coarsely injective quasi-metric spaces}\label{sub:Sym_coarsely_inj}
	
	We adapt the definition of coarsely injective metric spaces introduced in \cite{CheEst} (see also \cite{HaeHodPet}, where the name was adopted) and briefly mentioned in Remark \ref{rem:coarse_injective} to the realm of quasi-metric spaces.
	\begin{definition}\label{def:Sym_coarsely_inj}
	Let $(X,d)$ be a quasi-metric space, and $\delta\geq 0$. We say that $X$ is {\em $\delta$-$\Sym$-coarsely injective} (or simply {\em $\Sym$-coarsely injective}, if we do not need to specify the constant) if, for every family $\{(x_i,r_i,s_i)\}_{i\in I}\subseteq X\times\R_{\geq 0}\times\R_{\geq 0}$ satisfying $d(x_i,x_j)\leq r_i+s_j$ for every $i,j\in I$, we have
	$$\bigcap_{i\in I}(C_d(x_i,r_i+\delta)\cap C_{d^{-1}}(x_i,s_i+\delta))\neq\emptyset.$$
	\end{definition}

\begin{proposition}\label{prop:Sym_rough_isometries_and_Sym_coarsely_inj}
If $X$ and $Y$ are two $\Sym$-roughly isometric quasi-metric spaces, then $Y$ is $\Sym$-coarsely injective provided that so is $X$.
\end{proposition}
\begin{proof}
Let $\delta\geq 0$ such that $X$ is $\delta$-$\Sym$-coarsely injective. Let $\varepsilon\geq 0$ and $\varphi\colon X\to Y$ and $\psi\colon Y\to X$ be a pair of maps satisfying the property stated in Definition \ref{def:Sym_rough_isom}(a). Let now $\{(y_i,r_i,s_i)\}_{i\in I}\subseteq Y\times\R_{\geq 0}\times\R_{\geq 0}$ with the property that $d_Y(y_i,y_j)\leq r_i+s_j$, for every $i,j\in I$. Then $d_X(\psi(y_i),\psi(y_j))\leq r_i+s_j+\varepsilon$, and so there exists $z\in X$ such that $z\in C_{d_X}(\psi(y_i),r_i+\varepsilon/2+\delta)\cap C_{{d_X}^{-1}}(\psi(y_i),s_i+\varepsilon/2+\delta)$, for every $i\in I$. Hence, for every $i\in I$,
$$\begin{gathered}d_Y(y_i,\varphi(z))\leq d_Y(y_i,\varphi(\psi(y_i)))+d_Y(\varphi(\psi(y_i)),\varphi(z))\leq \varepsilon+d_X(\psi(y_i),z)+\varepsilon\leq r_i+(\delta+5\varepsilon/2),\\ \text{and, similarly, }\quad
d_Y(\varphi(z),y_i)\leq s_i+(\delta+5\varepsilon/2).\end{gathered}$$
Therefore, $Y$ is $(\delta+5\varepsilon/2)$-$\Sym$-coarsely injective.
\end{proof}

The following result, whose proof is analogous to that provided in \cite[Proposition 3.12]{ChaCheGenHirOsa}, justify the introduction of this class of quasi-metric spaces. 
\begin{theorem}\label{theo:Sym_coarsely_injective}
Let $X$ be a quasi-metric space. Then $X$ is $\delta$-$\Sym$-coarsely injective if and only if $X$ is $\delta$-$\Sym$-large in $Q(X)$.
\end{theorem}
\begin{proof}
Suppose that $X$ is $\delta$-$\Sym$-coarsely injective, and let $f\in Q(X)$. Since $d(x,y)\leq f_2(x)+f_1(y)$ for every $x,y\in X$, there exists 
$$z\in\bigcap_{x\in X}(C_d(x,f_2(x)+\delta)\cap C_{d^{-1}}(x,f_1(x)+\delta)).$$
Then, in particular, 
\begin{equation}\label{eq:*}\text{for every $x\in X$, $d(x,z)\leq f_2(x)+\delta$ and $d(z,x)\leq f_1(x)+\delta$.}\end{equation} 

Moreover, we claim that $d(x,z)-f_2(x)\geq-\delta$ and $d(z,x)-f_1(x)\geq-\delta$. We prove the first inequality while the other one can be similarly shown. Assume by contradiction $f_2(x)>d(x,z)+\delta$. In particular, $f_2(x)>0$. Define $\varepsilon=(f_2(x)-d(x,z)-\delta)/2$, and note that 
\begin{equation}\label{eq:**}f_2(x)>d(x,z)+\delta+\varepsilon.\end{equation}
According to \eqref{eq:Lips}, there exists $y\in X$ such that $f_2(x)<d(x,y)-f_1(y)+\varepsilon$. Thus, we have
$$f_1(y)+f_2(x)<d(x,y)+\varepsilon\leq d(x,z)+d(z,y)+\varepsilon<f_2(x)-\delta-\varepsilon+f_1(y)+\delta+\varepsilon=f_1(y)+f_2(x),$$
where the third inequality descends from \eqref{eq:*} and \eqref{eq:**}.
We have then obtain a contradiction. 

Therefore,
$$D^s(\mathfrak q(z),f)=\lvert\lvert d_{z,1}-f_1\rvert\rvert\vee\lvert\lvert d_{z,2}-f_2\rvert\rvert \leq\delta.$$

Suppose now that $X$ is $\delta$-$\Sym$-large in $Q(X)$, and let $\{(x_i,r_i,s_i)\}_{i\in I}\subseteq X\times\R_{\geq 0}\times\R_{\geq 0}$ such that $d(x_i,x_j)\leq r_i+s_j$ for every $i,j\in I$. Define $Y=\{x_i\mid i\in I\}$, $g_1\colon x_i\mapsto s_i$, and $g_2\colon x_i\mapsto r_i$, for every $i\in I$. Then $(g_1,g_2)\in P(Y)$, and thus there exists $f=(f_1,f_2)\in Q(X)$ such that $f_1(x_i)\leq s_i$ and $f_2(x_i)\leq r_i$ for every $i\in I$ (it is indeed enough to define $f=p_X\circ\sigma(g)$ according to Proposition \ref{prop:Q_of_subset}). Let $z\in X$ such that $$D^s(f,\mathfrak q(z))=\lvert\lvert d_{z,1}-f_1\rvert\rvert\vee\lvert\lvert d_{z,2}-f_2\rvert\rvert\leq\delta.$$ Thus, in particular, for every $i\in I$,
$$d(z,x_i)\leq f_1(x_i)+\delta\leq s_i+\delta\text{ and }d(x_i,z)\leq f_2(x)+\delta\leq r_i+\delta,$$
which implies $z\in C_d(x_i,r_i+\delta)\cap C_{d^{-1}}(x_i,s_i+\delta)$, and so $X$ is $\delta$-$\Sym$-coarsely injective.
\end{proof}

\begin{example}
Let $X$ be a bounded quasi-metric space, and let $\diam X=R$. Then $X$ is $R$-$\Sym$-coarsely injective. Therefore, $X$ is $R$-$\Sym$-large in $Q(X)$ and $\diam Q(X)\leq 3R$. 
\end{example}

\begin{corollary}\label{coro:Sym-roughly_iso_to_qhyper}
A quasi-metric space is $\Sym$-roughly isometric to a $q$-hyperconvex quasi-metric space if and only if it is $\Sym$-coarsely injective.
\end{corollary}
\begin{proof}
The `if' implication is an immediate consequence of Theorem \ref{theo:Sym_coarsely_injective}. Trivially a quasi-metric space is $q$-hyperconvex space if and only if it is $0$-$\Sym$-coarsely injective. Therefore, according to Proposition \ref{prop:Sym_rough_isometries_and_Sym_coarsely_inj}, if a quasi-metric space is $\Sym$-roughly isometric to a $q$-hyperconvex space, then it is $\Sym$-coarsely injective, and so $\Sym$-large in its $q$-hyperconvex hull.
\end{proof}
	
\begin{corollary}\label{coro:Sym-coarsely_injective_and_fixed_point}
Let $X$ be a $\delta$-$\Sym$-coarsely injective bounded quasi-metric space and $T\colon X\to X$ be a non-expansive map. Then there exists $x\in X$ with $d^s(x,T(x))\leq2\delta$.
\end{corollary}
\begin{proof}
Since $X$ is $\delta$-$\Sym$-coarsely injective, $X$ is $\delta$-$\Sym$-large in $Q(X)$. For the sake of simplicity, let us assume that $X$ is a subset of $Q(X)$. Since $Q(X)$ is $q$-hyperconvex, Theorem \ref{theo:q_hyper_iff_injective} implies that $T\colon X\to Q(X)$ can be extended to a non-expansive map $T^\prime\colon Q(X)\to Q(X)$. According to \cite[Theorem 3.3]{KunOta}, the set of fixed points $\Fix(T^\prime)$ of $T^\prime$ is non-empty (and also $q$-hyperconvex). Consider now the non-empty set $Y=C_{D^s}(\Fix(T^\prime),\delta)\cap X$. For every $y\in Y$, pick an element $f_y\in\Fix(T^\prime)$ such that $D^s(f_y,y)\leq\delta$. Then 
$$d^s(T(y),y)=D^s(T(y),y)\leq D^s(T^\prime(y),T^\prime(f_y))+D^s(T^\prime(f_y),f_y)+D^s(f_y,y)\leq 2\delta$$
since $T^\prime$ is non-expansive.
\end{proof}
In Corollary \ref{coro:Sym-coarsely_injective_and_fixed_point}, the parameter $\delta$ can be way smaller than the diameter of $X$, and so the claim is not trivial.


\begin{thebibliography}{99}
		
		\bibitem{Agy1} C.A. Agyingi, {\em On di-injective $T_0$-quasi-metric spaces}, Topology Appl. 228 (2017), 371--381.
		\bibitem{AgyHaiKun} C.A. Agyingi, P. Haihambo, H.-P.A. K\"unzi, {\em Tight extensions of $T_0$-quasi-metric spaces}, Logic, computation, hierarchies, 9--22, Ontos Math. Log., 4, De Gruyter, Berlin, 2014.
		\bibitem{AgyHaiKun1} C.A. Agyingi, P. Haihambo, H.-P.A. K\"unzi, {\em Endpoints in $T_0$-quasi-metric spaces}, Topology Appl. 168 (2014), 82--93.
		\bibitem{AroPan} N. Aronszajn, P. Panitchpakdi, {\em Extension of uniformly continuous transformations and hyperconvex metric spaces}, Pacific J. Math. 6 (1956), 405--439.
		\bibitem{BauRol} U. Bauer, F. Roll, {\em Gromov hyperbolicity, geodesic defect, and apparent pairs in Vietoris-Rips filtrations}, accepted to SoCG 2022, extended version: {\tt arXiv:2112.06781}.
		\bibitem{BhaSem}  R. Bhatia, P. \v{S}emrl, {\em Approximate isometries on Euclidean spaces}, Amer. Math. Monthly 104 (1997), no. 6, 497--504.
		\bibitem{BonSch} M. Bonk, O. Schramm, {\em Embeddings of Gromov hyperbolic spaces}, Geom. Funct. Anal. 10 (2000), no. 2, 266--306.
		\bibitem{BurBurIva} D. Burago, Y. Burago, S. Ivanov, {\em A Course in Metric Geometry}, AMS, 2001.
		\bibitem{CarMemRibSeg} G.E. Carlsson, F. M\'emoli, A. Ribeiro, S. Segarra, {\em Hierarchical quasi-clustering methods for asymmetric networks}, in ``Proceedings of the 31th International Conference on Machine Learning, ICML 2014'', 2014.
		\bibitem{CDFGBT} I. Castellano, D. Dikranjan, D. Freni, A. Giordano Bruno, D. Toller, {\em Intrinsic entropy for generalized quasimetric semilattices}, to appear in J. Algebra Appl.
		\bibitem{CasGioZav} I. Castellano, A. Giordano Bruno, N. Zava, {\em Weakly weighted generalised quasi-metric spaces and semilattices}, preprint.
		\bibitem{ChaCheGenHirOsa} J. Chalopin, V. Chepoi, A. Genevois, H. Hirai, D. Osajda, {\em Helly groups}, {\tt arXiv:2002.06895}.
		\bibitem{CheEst} V. Chepoi, B. Estellon, {\em Packing and covering $\delta$-hyperbolic spaces by balls}, in ``Approximation, Rondomization, and Combinatorial Optimization. Algorithms and Techniques'', 59--73, Springer, 2007.
		
		\bibitem{ChoMem1} S. Chowdhury, F. M\'emoli, {\em Metric structures on networks and applications}, in ``2015 53rd Annual Allerton Conference on Communication, Control, and Computing (Allerton)'', 1470--1472, Sept 2015.
		\bibitem{ChoMem2} S. Chowdhury, F. Mémoli, {\em Distances between directed networks and applications}, in ``2016 IEEE International Conference on Acoustics, Speech and Signal Processing (ICASSP)'', 6420--6424, IEEE, 2016.
		\bibitem{ChoMem3} S. Chowdhury, F. M\'emoli, {\em A functorial Dowker theorem and persistent homology of asymmetric networks}, Journal of Applied and Computational Topology 2 (1), 115--175.
		
		\bibitem{ChrLam} M. Chrobak, L. Lamore, {\em Generosity helps or an 11-competitive algorithm for three servers}, Journal of Algorithms 16 (1994), 234--263.
		\bibitem{ConKunOta} J. Conradie, H.-P.A. K\"unzi, O. Olela Otafudu, {\em The vector lattice structure on the Isbell-convex hull of an asymmetrically normed real vector space}, Topology Appl. 231 (2017), 92--112. 
		\bibitem{DezDez} M.M. Deza, E. Deza, {\em Encyclopedia of Distances}, Springer, Berlin, 2009.
		\bibitem{DGKZT} D. Dikranjan, A. Giordano Bruno, H.-P. K\"unzi, N. Zava, D. Toller, {\em Generalized quasi-metric semilattices}, Topology Appl. 309 (2022), Paper No. 107916, 35 pp.
		\bibitem{Dre} A.W.M. Dress, {\em Trees, tight extensions of metric spaces, and the cohomological dimension of certain groups: a note on combinatorial properties of metric spaces}, Adv. Math. 53 (1984), 321--402.
		\bibitem{DreHubMou} A.W.M. Dress, K.T. Huber, V. Moulton, {\em An explicit computation of the injective hull of certain finite metric spaces in terms of their associated Buneman complex}, Adv. in Math. 168 (2002), 1--28.
		\bibitem{DreMouTer} A.W.M. Dress, V. Moulton, W. Terhalle, {\em T-theory: an overview}, Europ. J. Combinatorics 17 (1996), 161--175.
		\bibitem{EdeHar} H. Edelsbrunner, J. Harer, {\em Computational topology: an introduction}, American Mathematical Soc., 2010.
		\bibitem{EspKha} R. Esp\'inola, M.A. Khamsi, {\em Introduction to hyperconvex spaces}, in Handbook of metric fixed point theory, 391--435, Kluwer Acad. Publ., Dordrecht, 2001.
		\bibitem{FleLin} P. Fletcher, W.F. Lindgren, {\em Quasi-Uniform Spaces}, Lecture Notes Pure Appl. Math., vol.77, Dekker, New York, 1982.
		\bibitem{Gou} J. Goubault-Larrecq, {\em Non-Hausdorff Topology and Domain Theory, Selected Topics in Point-Set Topology}, New Mathematical Monographs, Cambridge University Press (2013).
		\bibitem{Gro} M. Gromov, {\em Asymptotic invariants of infinite groups}, in ``Geometric group theory, Vol. 2'' (Sussex, 1991), 1--295, London Math. Soc. Lecture Note Ser., 182, Cambridge Univ. Press, Cambridge, 1993.
		\bibitem{HaeHodPet} T. Haettel, N. Hoda, H. Petyt, {\em Coarse injectivity, hierarchical hyperbolicity, and semihyperbolicity}, preprint {\tt arXiv:2009.14053}.
		\bibitem{Har} P. de la Harpe, \emph{Topics in geometric group theory}. Chicago Lectures in Math., the University of Chicago Press, Chicago (2000).
		\bibitem{Hau} F. Hausdorff, \textit{Set theory}, Second edition, translated from the German by John R. Aumann et al Chelsea Publishing Co., New York 1962 352 pp. 04.00
		\bibitem{HirKoi} H. Hirai, S. Koichi, {\em On tight spans for directed distances}, Ann. Comb. 16 (2012), 543--569.
		\bibitem{Isb} J.R. Isbell, {\em Six theorems about injective metric spaces}, Comment. Math. Helvetici 39 (1964), 65--76.
		\bibitem{KemKunOta} E. Kemajou, H.-P.A. K\"unzi, O. Olela Otafudu, {\em The Isbell-hull of a di-space}, Topology Appl. 159 (2012), 2463--2475.
		\bibitem{Kun} H.-P. Kunzi, {\em Nonsymmetric distances and their associated topologies: about the origins of basic ideals in the area of asymmetric topology}, in: Handbook of the History of General Topology, vol. 3, in: Hist. Topol., vol.3, Kluwer Acad. Publ., Dordrecht, 2001, pp.853--968.
		\bibitem{KunOta} H.-P.A. K\"unzi, O.O. Otafudu, {\em $q$-hyperconvexity in quasipseudometric spaces and fixed point theorems}, J. Funct. Spaces Appl. 2012, Art. ID 765903, 18 pp.
		\bibitem{KV94} H.-P. K\"unzi, V. Vajner, {\em Weighted Quasi-Metrics}, Annals of the New York Academy of Sciences 728 (1) (1994), 64--77.
		\bibitem{Lan}U. Lang, {\em Injective hulls of certain discrete metric spaces and groups}, J. Topol. Anal. 5 (2013), no. 3, 297--331.
		\bibitem{LanPavZus} U. Lang, M. Pav\'on, R. Z\"ust, {\em Metric stability of trees and tight spans}, Arch. Math. (Basel) 101 (2013), no. 1, 91--100.
		\bibitem{LimMemOku} S. Lim, F. M\'emoli, O.B. Okutan, {\em Vietoris-Rips Persistent Homology, Injective Metric Spaces, and The Filling Radius}, preprint, {\tt arXiv:2001.07588v3}.
		\bibitem{LimMemWanWanZho} S. Lim, F. M\'emoli, Z. Wan, Q. Wang, L. Zhou, {\em Some results about the Tight Span of spheres}, preprint, {\tt arXiv:2112.12646v1}.
		\bibitem{Mat} S. G. Matthews, {\em Partial metric topology}. Annals of the New York Academy of Sciences-Paper Edition 728 (1994), 183--197.
		\bibitem{Moe} A. Moezzi, {\em The Injective Hull of Hyperbolic Groups}, Dissertation ETH Zurich, No. 18860, 2010.
		\bibitem{Mon} N. Monod, {\em El\'ements de g\'eom\'etrie grossi\`ere}, ``Travail de dipl\^ome'' with M. Burger, University of Lausanne, Juillet 1997.
		\bibitem{OtaSab} O. Olela Otafudu, H. Sabao, {\em Set-valued maps and q-hyperconvex spaces}, J. Nonlinear Convex Anal. 18 (2017), no. 9, 1609--1617.
		\bibitem{Roe} J. Roe, {\em Lectures on Coarse Geometry}, Univ. Lecture Ser., vol. 31, American Mathematical Society, Providence RI, 2003.
		\bibitem{Sal} S. Salbany, {\em Injective objects and morphisms}, in: ``Categorical Topology and Its Relation to Analysis, Algebra and Combinatorics'', Prague, 1988, World Sci. Publ., Teaneck, NJ, 1989, 394--409.
		\bibitem{Sch_ext} M. P. Schellekens, {\em Extendible spaces}, Appl. Gen. Top. 3, 2 (2002) 169--184.
		\bibitem{Sch} M. P. Schellekens, {\em The correspondence between partial metrics and semivaluations}, Theoretical Computer Science 315 (2004), 135--149.
		\bibitem{SheZha} Y.-B. Shen, W. Zhao, {\em Gromov pre-compactness theorems for nonreversible Finsler manifolds}, Differential Geom. Appl. 28 (2010), no. 5, 565--581.
		\bibitem{Sin} R. C. Sine, {\em On linear contraction semigroups in sup norm spaces}, Nonlinear Anal. 3 (1979), 885--890.
		\bibitem{Soa}  P. Soardi, {\em Existence of fixed points for nonexpansive mappings in certain Banach lattices}, Proc. Amer. Math. Soc. 73 (1979), 25--29. 
		\bibitem{Wil} S. Willerton, {\em Tight spans, Isbell completions and semi-tropical modules}, Theory and Applications of Categories, Vol. 28, No. 22, 2013, pp. 696--732.
		\bibitem{Wils} W.A. Wilson, {\em On quasi-metric spaces}, Am. J. Math. 53(3) (1931) 675--684.
		\bibitem{Zav} N. Zava, {\em Generalisations of coarse spaces}, Topology Appl. 263 (2019), 230--256.
		\bibitem{Za_ent} N. Zava, {\em On a notion of entropy in coarse geometry}, Topol. Algebra Appl. 7 (2019), 48--68.
		\bibitem{Za_ent_action} N. Zava, {\em Algebraic entropy of endomorphisms of $M$-sets}, Topol. Algebra Appl. 9 (2021), no. 1, 53--71.
	\end{thebibliography}
\end{document}